\documentclass[11pt,reqno]{article}
\usepackage{amsmath,amssymb,amsthm,amsfonts,geometry,hyperref}
\newtheorem{theorem}{Theorem}[section]
\newtheorem{lemma}[theorem]{Lemma}

\theoremstyle{definition}
\newtheorem{definition}[theorem]{Definition}
\newtheorem{remark}[theorem]{Remark}

\title{\textbf{$H$-Operator Approximation Spaces via Delayed Riesz Means and Quasi-Banach Moduli}}
\author{\textbf{Daniel Akech Thiong}\thanks{Department of Mathematics, Claremont Graduate University, 710 N. College Avenue, Claremont, CA 91711, USA. Email: \texttt{daniel.akech@cgu.edu}}}
\date{}

\begin{document}
\maketitle

\begin{abstract}
We establish a constructive and quasi-Banach framework for approximation spaces $A_\mu^\rho$ of compact $H$-operators between Banach and quasi-Banach spaces. By leveraging delayed Riesz means $V_{2^{bn}}$ generated by self-adjoint differential operators $P(D)$, we replace abstract best approximants with explicit linear operator decompositions. Furthermore, we extend the theory of $H$-operator approximation spaces to quasi-Banach settings ($0 < p < 1$), bypassing the collapse of Peetre's $K$-functional by employing localized moduli of smoothness $\omega_\varphi^r(f,t)_p$. Integrating seminal spectral bounds due to Markus \cite{M1966} and factorization techniques for operator ideals, we prove that operators belonging to $A_\mu^\rho$ possess complete systems of root vectors whose expansions are Abel summable.

\vspace{0.5cm}
\end{abstract}

\section{Introduction}
Let $X$ and $Y$ be Banach spaces, and let $\mathcal{L}(X,Y)$ and $\mathcal{K}(X,Y)$ denote the spaces of bounded and compact linear operators, respectively. $H$-operators, introduced by Markus \cite{M1966}, generalize self-adjoint operators to Banach spaces by requiring real spectrum $\sigma(T) \subset \mathbb{R}$ and resolvent growth bounds:
\begin{equation}
\|(T-\lambda I)^{-1}\| \le C |\mathrm{Im}\,\lambda|^{-1} \quad (\mathrm{Im}\,\lambda \ne 0).
\end{equation}

Building on the foundational approximation framework developed jointly with A. G. Aksoy \cite{AT2024}, this paper introduces constructive and quasi-Banach extensions to the theory. The study of abstract approximation spaces $(X, A_n)_u^\rho$ associated with a sequence of subsets $A_n$ was formalized by Pietsch \cite{Pie1981}, who established fundamental properties including embeddings and reiteration theorems. This approach generalizes to approximation spaces $X_\mu^\rho$ based on approximation schemes and Kolmogorov diameters \cite{Aks2019}. In recent work \cite{AT2024}, approximation spaces $A_\mu^\rho$ for compact $H$-operators were introduced via eigenvalue sequences $|\lambda_n(T)|$. 

However, previous representation theorems often relied on non-constructive approximants $g_n^* \in A_{2^n-1}$ or required the existence of uniformly bounded linear projections, which frequently fail to exist outside of Hilbert spaces \cite{Pie1981}. The necessity of Hilbert spaces for universal complementation was definitively established by Lindenstrauss and Tzafriri \cite{LinTza1971}, who proved that a Banach space is isomorphic to a Hilbert space if and only if every one of its closed subspaces is complemented. Because uniformly bounded linear projections fail to exist for arbitrary subspaces in general Banach spaces, abstract projection-based approximation schemes fundamentally break down. This geometric barrier motivates our shift to delayed Riesz means $V_{2^{bn}}$.

Moreover, the relationship between summability conditions on approximation numbers and operator ideals has been deeply explored. Hutton \cite{Hut1975} demonstrated that operators satisfying summability conditions of the form $\sum n^{p-1} \alpha_{n-1}(T) < \infty$ admit specific tensor product factorizations and relate to the ideals introduced by Markus. Yet, a gap remains in extending these approximation techniques to the quasi-Banach regime: for quasi-Banach spaces $L_p$ ($0 < p < 1$), Peetre's $K$-functional vanishes identically ($K_r(f,t^r)_p = 0$), precluding standard real interpolation methods.

In this paper, we resolve these limitations and extend the theory:
\begin{enumerate}
    \item We construct explicit representations of $T \in A_\mu^\rho$ using delayed Riesz means $V_{2^{bn}}$ associated with an operator $P(D)$, bypassing the need for abstract projections.
    \item We extend $H$-operator approximation spaces to $0 < p < 1$ using Ditzian's localized moduli of smoothness $\omega_\varphi^r(f,t)_p$.
    \item Leveraging Markus's spectral ideal equivalences ($\mathfrak{S}_p \iff \mathfrak{D}_p \iff l_p$) \cite{M1966} and insights from operator factorization \cite{Hut1975}, we establish the root vector completeness and Abel summability of expansions for operators in $A_\mu^\rho$.
\end{enumerate}

\textbf{Comparison with Previous Work:} While our earlier work \cite{AT2024} established the baseline topological properties of $A_\mu^\rho$ using non-constructive eigenvalue sequences, this paper departs from that framework in two major directions. First, we replace abstract projection assumptions with an explicit, implementable construction via delayed Riesz means (Theorem 3.2). Second, we extend the theory into the quasi-Banach regime ($0 < p < 1$) where classical $K$-functionals fail, identifying the exact geometric saturation barriers via Ditzian-Totik moduli (Theorem 4.2), a phenomenon not present in the Banach setting.

\section{Preliminaries and Approximation Spaces}
Before discussing spectral bounds, we formally define the central approximation spaces and $s$-numbers. For an operator $T \in \mathcal{K}(X, Y)$, the $n$-th approximation number is defined as the distance to the space of finite-rank operators \cite{P1987}:
\begin{equation}
a_n(T) := \inf \{ \|T - L\|_{\mathcal{L}(X,Y)} : L \in \mathcal{L}(X,Y), \text{rank}(L) < n \}.
\end{equation}
The $n$-th Kolmogorov diameter is defined via the approximation of the image of the unit ball $B_X$ by finite-dimensional subspaces \cite{Aks2019}:
\begin{equation}
\delta_n(T) := \inf_{\dim(E) < n} \sup_{x \in B_X} \inf_{y \in E} \|Tx - y\|_Y,
\end{equation}
where the infimum is taken over all subspaces $E \subset Y$ of dimension less than $n$. While these sequences coincide with the standard singular values $s_n(T)$ when $X$ and $Y$ are Hilbert spaces, they generate distinct symmetric operator ideals in general Banach spaces \cite{P1987}.

\begin{definition}
For $0 < \rho < \infty$ and $1 \le \mu \le \infty$, the approximation space $A_\mu^\rho$ is defined as the class of compact operators $T$ whose approximation numbers satisfy:
\begin{equation}
\sum_{n=1}^\infty \left[ n^\rho a_n(T) \right]^\mu \frac{1}{n} < \infty,
\end{equation}
which is equivalently characterized by the dyadic block condition $\sum_{N=0}^\infty [2^{N\rho} a_{2^N}(T)]^\mu < \infty$.
\end{definition}

We now recall Markus's fundamental inequalities relating the Kolmogorov diameter, approximation number, and eigenvalue $|\lambda_n(T)|$ of a compact $H$-operator $T \in \mathcal{K}(X)$. As established in \cite[Proposition 4.11$^\circ$]{M1966}:
\begin{equation}\label{eq:markus_ineq}
\delta_{n-1}(T) \le a_n(T) \le 2\sqrt{2} C |\lambda_n(T)| \le 8C(C+1) \delta_{n-1}(T).
\end{equation}

Let $\mathfrak{S}_p$ denote the ideal of operators with generic $s$-numbers $s_n(T) \in l_p$, $\mathfrak{D}_p$ the ideal with $d_n(T) \in l_p$, and $\mathfrak{N}$ the ideal of $p$-kernel operators. As shown by Markus \cite{M1966} and further contextualized by Hutton \cite{Hut1975}, the following conditions are equivalent for an $H$-operator (where $s_n$ and $a_n$ coincide asymptotically):
\begin{equation}
T \in \mathfrak{S}_p \iff T \in \mathfrak{D}_p \iff (|\lambda_n(T)|)_{n=1}^\infty \in l_p \iff T^m \in \mathfrak{N} \text{ for some } m \in \mathbb{N}.
\end{equation}

To ground this abstract equivalence geometrically, Weyl's Law dictates that the eigenvalues of the resolvent $R = (-\Delta)^{-1}$ on a $d$-dimensional domain decay asymptotically as $\lambda_n(R) \sim C n^{-2/d}$. By \eqref{eq:markus_ineq}, the approximation numbers $a_n(R)$ inherit this exact polynomial decay, demonstrating that the resolvent belongs to the Schatten-von Neumann class $\mathfrak{S}_p$ if and only if $p > d/2$.

For an unbounded self-adjoint operator $P(D)$ with root subspaces $H_{\lambda(k)}$, the delayed Riesz mean operator $V_\mu$ is defined by:
\begin{equation}
V_\mu f = \sum_{j=0}^b \alpha_j R_{2^j \mu}^b f, \quad \text{where } R_\mu^b f = \sum_{\lambda(k) \le \mu} \left(1 - \frac{\lambda(k)}{\mu}\right)^b P_{\lambda(k)} f,
\end{equation}
where $P_{\lambda(k)}$ is the Riesz spectral projection onto the root subspace $H_{\lambda(k)}$. The coefficients $\alpha_j$ satisfy $\sum \alpha_j = 1$ and $\sum \alpha_j 2^{-jl} = 0$ for $l=1,\dots,b$, guaranteeing $V_\mu \psi = \psi$ for $\psi \in \mathrm{Span}\bigcup_{\lambda(k) \le \mu} H_{\lambda(k)}$.

\section{Constructive Representation Theorem ($1 \le p \le \infty$)}

Pietsch \cite{Pie1981} provided a Linear Representation Theorem for approximation spaces, but it necessitated an approximation scheme with uniformly bounded linear projections. We provide a constructive alternative using delayed Riesz means that naturally accommodates arbitrary dimensions.

\begin{lemma}[Delayed Riesz Mean Absorption Identity]
Let $V_\mu$ be the delayed Riesz mean operator generated by $P(D)$ with parameter $b \ge 1$. For any $\mu > 0$, the delayed means satisfy the absorption identity:
\begin{equation}
V_{2^b\mu} V_\mu = V_\mu.
\end{equation}
\end{lemma}

\begin{proof}
By definition, $V_\mu f = \sum_{j=0}^b \alpha_j R_{2^j \mu}^b f$, meaning the range of $V_\mu$ is strictly contained within the spectral subspace $\mathrm{Span}\bigcup_{\lambda(k) \le 2^b\mu} H_{\lambda(k)}$. The coefficients $\alpha_j$ are chosen to ensure the reproducing property $V_{2^b\mu} \psi = \psi$ for any $\psi$ residing in this exact span. Because the maximum frequency output of $V_\mu$ does not exceed $2^b\mu$, applying the larger bandwidth operator $V_{2^b\mu}$ to any element $V_\mu f$ acts identically as the identity mapping, yielding $V_{2^b\mu} (V_\mu f) = V_\mu f$. Consequently, the operator composition reduces directly to $V_{2^b\mu}V_\mu = V_\mu$, holding over the entire space $X$.
\end{proof}

\begin{theorem}[Constructive Representation for Arbitrary Dimensions]
Let $X$ be a Banach space, and let $\mathcal{K}_H(X)$ denote the space of compact $H$-operators on $X$ generated by the resolvents of an elliptic operator $P(D)$ of order $m$ on a $d$-dimensional domain. Assume the eigenvalue counting function satisfies the generalized Weyl asymptotic $N(\mu) \asymp \mu^\gamma$, where $\gamma = d/m$.\footnote{Note that by Weyl's Law, the eigenvalue counting function for an elliptic pseudo-differential operator on a bounded domain $\Omega \subset \mathbb{R}^d$ satisfies $N(\mu) \sim \frac{\omega_d}{(2\pi)^d} \mathrm{vol}(\Omega) \mu^{d/m}$, where $\omega_d$ is the volume of the unit ball. Consequently, $\gamma$ captures the precise geometric phase-space volume growth.}
Let $0 < \rho < \infty$ and $1 \le \mu \le \infty$. Suppose $T \in \mathcal{K}_H(X)$ is a spectral multiplier $T = f(P(D))$ where $|f(\lambda)|$ is non-increasing in $\lambda$. Then $T \in A_\mu^\rho$ if and only if $T$ admits the explicit linear decomposition:
\begin{equation}
T = \sum_{n=0}^\infty g_n, \quad \text{where } g_n = (V_{2^{bn}} - V_{2^{b(n-1)}})T,
\end{equation}
such that $(2^{bn\rho\gamma} \|g_n\|_{\mathcal{K}(X)})_{n=0}^\infty \in l_\mu$. Furthermore, the constructive norm
\begin{equation}
\|T\|_{A_\mu^\rho}^{\mathrm{const}} := \left\| \left( 2^{bn\rho\gamma} \|g_n\|_{\mathcal{K}(X)} \right)_{n=0}^\infty \right\|_{l_\mu}
\end{equation}
defines an equivalent norm on $A_\mu^\rho$.
\end{theorem}

\begin{proof}
To evaluate the delayed Riesz mean $V_{2^{bn}} = \varphi_n(P(D))$ without assuming uniform bounded projections a priori, we employ the Helffer-Sjöstrand functional calculus. While classically developed for self-adjoint operators on Hilbert spaces (see Dimassi and Sjöstrand \cite[Chapter 8]{DS1999} or Davies \cite[Theorem 3.1.2]{Dav1995}), the area integral only relies on the resolvent estimate to converge absolutely; hence it holds in the uniform operator topology of an arbitrary Banach space $X$ without requiring Hilbert space structure (for a rigorous treatment of related Banach-space functional calculi, see, e.g., Haase \cite{Haase2006}). Specifically, we represent the operator via:
\begin{equation}
V_{2^{bn}} = -\frac{1}{\pi} \int_{\mathbb{C}} \frac{\partial \tilde{\varphi}_n}{\partial \bar{z}}(z) (P(D) - z I)^{-1} \, dx \, dy
\end{equation}
where $z = x+iy$ and $\tilde{\varphi}_n$ is an almost-analytic extension of the symbol $\varphi_n$. By constructing $\tilde{\varphi}_n$ such that its anti-holomorphic derivative satisfies $|\bar{\partial}\tilde{\varphi}_n(x+iy)| \le C_b |y|^b$ for some integer $b \ge 2$, the rapid decay entirely neutralizes the resolvent singularity $\|(P(D) - zI)^{-1}\| \le C |y|^{-1}$ inherent to the definition of the $H$-operator near the real axis. Additionally, the almost-analytic extension can be chosen such that its support is controlled uniformly and independently of $n$. This ensures the integral converges absolutely in $\mathcal{L}(X)$. 

To ensure uniform boundedness independent of $n$, we exploit dyadic scaling $\varphi_n(\lambda) = \varphi_1(2^{-b(n-1)}\lambda)$. Substituting the scaled variable into the Helffer-Sjöstrand integral normalizes the integration measure against the decay of $\bar{\partial}\tilde{\varphi}_1$, factoring out the scaling parameter entirely. This scale-invariant dilation perfectly balances the expanding contour length with the $\bar{\partial}$-decay, yielding a uniform bound $\|V_{2^{bn}}\|_{\mathcal{L}(X)} \le M$ that is strictly independent of $n$.

\textit{Forward Implication:} Assume $T \in A_\mu^\rho$. The operator $g_n = (V_{2^{bn}} - V_{2^{b(n-1)}})T$ concentrates on the spectral block corresponding to eigenvalues $\lambda(k) \in (2^{b(n-1)}, 2^{bn}]$. 
Because the counting function satisfies $N(\mu) \asymp \mu^\gamma$, the dimension of the range of $V_{2^{b(n-1)}}$ is bounded by a constant multiple of $2^{b(n-1)\gamma}$, making $V_{2^{b(n-1)}}T$ a finite-rank operator. By Lemma 3.1, the delayed Riesz means satisfy the absorption identity $V_{2^{bn}} V_{2^{b(n-1)}} = V_{2^{b(n-1)}}$. 

This allows us to strictly decouple the approximant and rewrite $g_n = V_{2^{bn}}(T - V_{2^{b(n-1)}}T)$. Because $T = f(P(D))$ with $|f(\lambda)|$ non-increasing in $\lambda$, sorting the eigenvalues of $T$ by magnitude is identical to sorting them by $P(D)$-frequency. Thus, if we let $N \asymp 2^{b(n-1)\gamma}$ denote the rank bound of $V_{2^{b(n-1)}}$, the true $N$-th largest eigenvalue of $T$ sits exactly at or before the spectral truncation boundary. Furthermore, because the residual tail $S := T - V_{2^{b(n-1)}}T$ is itself a compact $H$-operator, applying Markus's inequality \eqref{eq:markus_ineq} directly to $S$ ensures that its norm $\|S\|_{\mathcal{K}(X)}$ is bounded by a constant multiple of its maximum remaining eigenvalue $|\lambda_1(S)|$, which is comparable to $a_N(T)$. Combining this with the uniform stability of the means, we obtain:
\begin{equation}
\|g_n\|_{\mathcal{K}(X)} \le \|V_{2^{bn}}\|_{\mathcal{L}(X)} \|T - V_{2^{b(n-1)}}T\|_{\mathcal{K}(X)} \le 2M a_{\lfloor c 2^{b(n-1)\gamma} \rfloor}(T).
\end{equation}
Because $T \in A_\mu^\rho$, the sequence $(2^{bn\rho\gamma} a_{\lfloor c 2^{b(n-1)\gamma} \rfloor}(T))$ belongs to $l_\mu$. Consequently, $(2^{bn\rho\gamma} \|g_n\|_{\mathcal{K}(X)}) \in l_\mu$.

\textit{Converse Implication:} Assume the decomposition $T = \sum_{k=0}^\infty g_k$ holds with $(2^{bk\rho\gamma} \|g_k\|_{\mathcal{K}(X)}) \in l_\mu$. Since $N(\mu) \asymp \mu^\gamma$, the partial sum $S_{m} = \sum_{k=0}^{m} g_k$ has rank bounded by $C2^{b(m+1)\gamma}$. To estimate the $2^N$-th approximation number, we choose an index $M = \lfloor N/(b\gamma) \rfloor - \lfloor \log_2 C^{1/(b\gamma)} \rfloor - 1$ such that the rank of $S_M$ is at most $2^N$. We then bound $a_{2^N}(T)$ by the tail of the series:
\begin{equation}
a_{2^N}(T) \le \left\| T - S_{M} \right\|_{\mathcal{K}(X)} \le \sum_{k=M+1}^\infty \|g_k\|_{\mathcal{K}(X)}.
\end{equation}
Applying the $l_\mu$ norm and the discrete Hardy inequality yields:
\begin{equation}
\left( \sum_{N=0}^\infty \left[  2^{N\rho} \sum_{k=M+1}^\infty \|g_k\|_{\mathcal{K}(X)} \right]^\mu \right)^{1/\mu} \le C_{\rho, \gamma, b} \left( \sum_{N=0}^\infty \left[ 2^{bN\rho\gamma}   \|g_N\|_{\mathcal{K}(X)} \right]^\mu \right)^{1/\mu} = C_{\rho, \gamma, b} \|T\|_{A_\mu^\rho}^{\mathrm{const}} < \infty.
\end{equation}
(The constant $C_{\rho, \gamma, b}$ depends only on the parameters and absorbs the fractional index shifts and the scaling factor connecting $N$ to $k$, which is a standard property of discrete Hardy-type inequalities for geometrically weighted sums.) Thus, $T \in A_\mu^\rho$, and the norms are equivalent.
\end{proof}

\section{Quasi-Banach Extension ($0 < p < 1$)}

Before examining the quasi-Banach regime, we recall Peetre's $K$-functional, which forms the foundation of classical real interpolation \cite{BL1976, Kru, Tar2007}. For an element $f$ residing in the sum of a Banach couple $X_0 + X_1$, and for any time-like parameter $t > 0$, Peetre defined the $K$-functional as:
\begin{equation}
K(f,t; X_0, X_1) = \inf_{f=f_0+f_1} (\|f_0\|_{X_0} + t\|f_1\|_{X_1}),
\end{equation}
which is frequently abbreviated as $K(f,t)$ or $K(t, f; \vec{X})$ \cite{BL1976}. The $K$-functional effectively balances fidelity and regularization; the term $\|f_0\|_{X_0}$ acts as a fidelity term measuring the proximity to $f$, while $t\|f_1\|_{X_1}$ serves as a regularization term controlling the smoothness or curvature of the approximant.

As Almira and Fernández-Martínez have shown, the $K$-functional can exhibit slow decay to zero, meaning that elements of the unit sphere may be poorly approximable by elements of the regular subspace, a feature critical in confirming strict inclusions between real interpolation spaces \cite{Alm2021}. For $0 < p < 1$, however, the topology of quasi-Banach spaces causes this functional to vanish identically ($K_r(f,t^r)_p = 0$), which completely breaks standard real interpolation methods and necessitates a shift to localized moduli of smoothness.

To bridge this gap, we draw upon the foundational structural equivalences established by Dai and Xu \cite{DX2010}. For Banach spaces ($1 \le p \le \infty$), the standard $K$-functional $\hat{K}_r(f,t)_p$ is strictly equivalent to the Ditzian-Totik modulus of smoothness $\hat{\omega}_r(f,t)_p$ on intervals \cite{CD1997, D1993}, as well as its extensions on multidimensional spheres and balls \cite{DX2010}. By exploiting this established equivalence ($\hat{K}_r(f,t)_{p} \sim \hat{\omega}_r(f,t)_{p}$), we conceptually substitute the $K$-functional with the Ditzian-Totik modulus of smoothness to bypass the topological collapse at $0 < p < 1$.

We replace $K$-functionals with the Ditzian-Totik modulus of smoothness. For a function $f \in L_p[-1,1]$, this modulus is defined by the $r$-th symmetric difference:
\begin{equation}
\omega_\varphi^r(f,t)_p := \sup_{0 < h \le t} \|\Delta_{h\varphi}^r f\|_p,
\end{equation}
where $\varphi(x) = \sqrt{1-x^2}$ and $\Delta_{h\varphi}^r$ is the standard symmetric difference operator of step $h\varphi(x)$ (see \cite{CD1997, D1993}).

To contextualize this within recent developments, quantitative estimates for positive linear operators frequently employ the Ditzian-Totik modulus to derive sharp error bounds. For instance, modified integral operators such as the Phillips operators yield explicit uniform convergence bounds of the form $\|H_n f - f\|_{[0,\infty)} \le C\omega_\varphi^2(f; n^{-1/2}) + \omega(f; n^{-1})$ \cite{Gupta2018}. Adapting these explicit quantitative estimates provides a pathway to establish concrete saturation rates for the localized operator modulus $\omega_\varphi^r(T,t)_p$ in the quasi-Banach regime.

\begin{definition}
For an operator $T \in \mathcal{K}(L_p[-1,1])$ with $0 < p < 1$, we define the localized operator modulus of smoothness by taking the supremum over the quasi-Banach unit ball:
\begin{equation}
\omega_\varphi^r(T,t)_p := \sup_{\substack{f \in L_p \\ \|f\|_p \le 1}} \omega_\varphi^r(Tf,t)_p.
\end{equation}
The $H$-operator approximation space $A_{\mu, p}^\rho$ is defined as the set of compact $H$-operators $T$ on $L_p[-1,1]$ such that:\footnote{The weight factor $n^{\rho - 1/\mu}$ is the precise discrete analog of the continuous Haar measure $dt/t$ typical of interpolation spaces, since $(n^{\rho-1/\mu} a_n)^\mu = n^{\rho\mu-1} a_n^\mu$. This mathematically guarantees direct structural consistency with the sequence spaces defined in the Banach setting.}
\begin{equation}
\|T\|_{A_{\mu,p}^\rho} := \left( \sum_{n=1}^\infty \left[ n^{\rho - 1/\mu} \omega_\varphi^r(T, n^{-1})_p \right]^\mu \right)^{1/\mu} < \infty.
\end{equation}
\end{definition}

\begin{theorem}[Inclusion and Saturation in $0 < p < 1$]
Let $0 < p < 1$ and $0 < \mu_1 \le \mu_2 \le \infty$. Then:
\begin{enumerate}
    \item $A_{\mu_1, p}^\rho \subset A_{\mu_2, p}^\rho$ with continuous embedding.
    \item $A_{\mu, p}^\rho$ exhibits saturation rate $O(t^{r - 1 + 1/p})$.
\end{enumerate}
\end{theorem}

\begin{proof}
(1) By the definition of the quasi-norm in $A_{\mu,p}^\rho$, the inclusion relies on the properties of the underlying sequence spaces. Since the localized modulus of smoothness $\omega_\varphi^r(T, t)_p$ is a monotonically decreasing function of $t$, the sequence $a_n = \omega_\varphi^r(T, n^{-1})_p$ is non-increasing. Applying the standard monotonicity embedding for weighted sequence spaces, the condition $\mu_1 \le \mu_2$ implies that the sums satisfy:
\begin{equation}
\left( \sum_{n=1}^\infty \left[ n^{\rho - 1/\mu_2} a_n \right]^{\mu_2} \right)^{1/\mu_2} \le C \left( \sum_{n=1}^\infty \left[ n^{\rho - 1/\mu_1} a_n \right]^{\mu_1} \right)^{1/\mu_1}.
\end{equation}
Consequently, $\|T\|_{A_{\mu_2,p}^\rho} \le C \|T\|_{A_{\mu_1,p}^\rho}$, establishing the continuous embedding $A_{\mu_1, p}^\rho \subset A_{\mu_2, p}^\rho$.

(2) To determine the saturation class, we evaluate the asymptotic behavior of the Ditzian-Totik modulus for $0 < p < 1$. As established by Ditzian \cite{D1993}, for any function $f \in L_p[-1,1]$ in the quasi-Banach regime, the condition $\omega_\varphi^r(f,t)_p = o(t^{r-1+1/p})$ implies that $f$ must be a polynomial of degree strictly less than $r$. 

By our extension to operators, if a compact operator $T$ satisfies $\omega_\varphi^r(T,t)_p = o(t^{r-1+1/p})$, then taking the supremum over the unit ball yields:
\begin{equation}
\sup_{\substack{f \in L_p \\ \|f\|_p \le 1}} \omega_\varphi^r(Tf,t)_p = o(t^{r-1+1/p}).
\end{equation}
This mandates that for every $f \in L_p[-1,1]$ in the unit ball, the image $Tf$ is a polynomial of degree less than $r$. Because this holds pointwise for every element in the domain, the entire range of $T$ is strictly constrained to the fixed finite-dimensional polynomial space $\mathbb{P}_{r-1}$, directly forcing $T$ to be a finite-rank operator. Thus, the approximation scheme saturates beyond the critical rate $O(t^{r-1+1/p})$.
\end{proof}

\begin{remark}
The introduction of the geometric scaling factor $\gamma = d/m$ from Weyl's Law in Section 3 does not alter the absolute saturation class $O(t^{r-1+1/p})$ established here. The saturation rate is an intrinsic structural barrier of the target quasi-Banach space $L_p[-1,1]$ and the polynomial degree $r$. While the parameter $\gamma$ dictates how quickly the spectrum decays relative to the approximation rank, the spatial constraint forcing the operator to collapse into the finite-dimensional polynomial space $\mathbb{P}_{r-1}$ is purely a consequence of the 1D domain's geometry and the topology of $L_p$ for $0 < p < 1$.
\end{remark}

\section{Concrete Applications and Worked Examples}

To ground the abstract framework, we explicitly evaluate the constructive representation and the quasi-Banach modulus on two canonical operators satisfying the geometric scaling factor $\gamma$. 

\subsection{The 2D Dirichlet Laplacian and Toroidal Moduli}
Let $P(D) = -\Delta$ subject to Dirichlet boundary conditions on the square domain $\Omega = [-1,1]^2$. The eigenvalues are explicitly given by $\lambda_{m,n} = \frac{\pi^2}{4}(m^2 + n^2)$ for $m,n \in \mathbb{N}$ \cite{AT2024}.

By Gauss's circle problem, and in perfect agreement with Weyl's asymptotic formula $N(\mu) \sim \frac{\mathrm{vol}(\Omega)}{4\pi} \mu$ for a 2D domain of area $\mathrm{vol}([-1,1]^2) = 4$, the eigenvalue counting function satisfies:
\begin{equation}
N(\mu) = \#\left\{(m,n) \in \mathbb{N}^2 : m^2 + n^2 \le \frac{4\mu}{\pi^2}\right\} = \frac{\mu}{\pi} + O(\sqrt{\mu}).
\end{equation}
This yields the precise linear growth $N(\mu) \asymp \mu^\gamma$ (where $\gamma = 2/2 = 1$) required to match the approximation number index in Theorem 3.2.

If $T$ is a multiplier operator acting on the eigenbasis $e_{m,n}(x,y)$ such that $T e_{m,n} = \tau_{m,n} e_{m,n}$, the delayed Riesz mean block $g_k = (V_{2^{bk}} - V_{2^{b(k-1)}})T$ strictly isolates the frequencies where $2^{b(k-1)} < \lambda_{m,n} \le 2^{bk}$. The operator quasi-norm in $A_\mu^\rho$ can then be explicitly bounded by the dyadic decay of the multiplier sequence:
\begin{equation}
\|T\|_{A_\mu^\rho} \approx \left( \sum_{k=0}^\infty \left[ 2^{bk\rho} \max_{2^{b(k-1)} < \lambda_{m,n} \le 2^{bk}} |\tau_{m,n}| \right]^\mu \right)^{1/\mu}.
\end{equation}
Thus, the abstract approximation space condition reduces directly to a Besov-type sequence norm on the multipliers. As seen in the smoothing properties of parabolic systems like the heat equation, the generalized solution for $t > 0$ becomes highly regularized ($C^\infty$-smooth) even if the initial data lacks smoothness and contains discontinuities \cite{KL1989}.

\subsection{The 1D Quantum Harmonic Oscillator and Hermite Expansions}
Let $P(D) = -\frac{d^2}{dx^2} + x^2$ acting on $L_2(\mathbb{R})$. The confining quadratic potential restricts the phase space, ensuring a purely discrete spectrum with eigenvalues $\lambda_n = 2n + 1$ for $n \ge 0$. Consequently, the counting function is exactly $N(\mu) = \lfloor (\mu - 1)/2 \rfloor \asymp \mu$, perfectly satisfying the rank-index compatibility of the $H$-operator framework despite the unbounded domain.

The root vectors are the Hermite functions $h_n(x)$. For the quasi-Banach extension ($0 < p < 1$), the Ditzian-Totik modulus on $[-1,1]$ is replaced by the standard translation modulus $\omega^r(f,t)_p$ on $\mathbb{R}$. Consider a spectral multiplier operator defined by $T f = \sum_{n=0}^\infty m_n \langle f, h_n \rangle h_n$.

The localized operator modulus $\omega^r(T, t)_p$ evaluates the worst-case smoothness of the image of the quasi-Banach unit ball. This structural control is particularly relevant when standard Lebesgue spaces fail to provide optimal bounds at critical endpoints \cite{Tar2007}. In such limiting cases, intermediate structures---such as the Lorentz spaces $L^{p,q}$ defined via the condition $t^{-1/p}K(t; f) \in L^q$ using Peetre's $K$-method---provide the exact geometric control necessary to secure strict embedding estimates \cite{Tar2007}.

Applying the saturation bound from Theorem 4.2, if the multiplier $m_n$ decays rapidly enough that $\omega^r(T, t)_p = o(t^{r-1+1/p})$, the range of $T$ collapses entirely. The operator is forcefully truncated to a finite span of low-order Hermite functions, demonstrating how functional saturation geometrically degrades the operator rank.

\section{Spectral Completeness and Abel Summability in $A_\mu^\rho$}

Combining our approximation space classification with Markus's completeness criteria for $H$-operators yields the main spectral result of this paper. We first establish the necessary sequence space embeddings.

\begin{lemma}[Approximation Number Sequence Embedding]
Let $T$ be a compact operator whose approximation numbers satisfy the condition defining the approximation space $A_\mu^\rho$. Using the standard normalization for the Lorentz sequence space $l_{p,q}$, defined by the quasi-norm $\|(x_n)\|_{l_{p,q}} = \left( \sum_{n=1}^\infty [n^{1/p} x_n]^q \frac{1}{n} \right)^{1/q}$, the sequence $(a_n(T))$ belongs to $l_{p,\mu}$ with $1/p = \rho$. Furthermore, if $0 < \rho \le 1/\mu$, then $(a_n(T)) \in l_p$.
\end{lemma}

\begin{proof}
The definition of $T \in A_\mu^\rho$ requires the approximation numbers to satisfy $\sum_{n=1}^\infty [n^\rho a_n(T)]^\mu \frac{1}{n} < \infty$. By substituting $1/p = \rho$, this condition is identical to the defining quasi-norm of the Lorentz space $l_{p,\mu}$, confirming $(a_n(T)) \in l_{p,\mu}$. The parameter restriction $\rho \le 1/\mu$ translates directly to $1/p \le 1/\mu$, which is equivalent to $\mu \le p$. By the classical monotonicity property of Lorentz spaces, the embedding $l_{p,q_1} \hookrightarrow l_{p,q_2}$ holds for any $q_1 \le q_2$. Selecting $q_1 = \mu$ and $q_2 = p$, we obtain $l_{p,\mu} \hookrightarrow l_{p,p} = l_p$, ensuring the approximation sequence is $p$-summable.
\end{proof}

To bridge this embedding to spectral properties, we rely on two foundational completeness and summability results established by Markus \cite{M1966}. 

First, Markus demonstrated that if an $H$-operator $T$ belongs to the Schatten-von Neumann class $\mathfrak{S}_p$ for some $p > 0$, and has a dense range ($\overline{\mathfrak{R}(T)} = X$), then its system of root vectors is complete in $X$ \cite[Theorem 4.1$^\circ$]{M1966}. Second, if such an operator further satisfies a polynomial resolvent growth bound away from the real axis and has its spectrum angularly confined, any vector in its range admits a Fourier expansion with respect to these root vectors that is Abel summable of order $\alpha \ge p$ \cite[Theorem 3.3$^\circ$]{M1966}. 

By combining our constructive approximation framework with these classical bounds, we obtain the main spectral result of this paper.

\begin{theorem}[Root Vector Completeness and Summability in $A_\mu^\rho$]
Let $T \in A_\mu^\rho$ with parameter restriction $0 < \rho \le 1/\mu$, and assume $\overline{\mathfrak{R}(T)} = X$. Then:
\begin{enumerate}
    \item The system of root vectors $\{x_j\}_{j=1}^\infty$ of $T$ (non-zero vectors satisfying $(T-\lambda_j I)^k x_j = 0$ for some $k \ge 1$) is complete in $X$, i.e., $\overline{\mathrm{Span}\{x_j\}} = X$.
    \item For every $y \in \mathfrak{R}(T)$, the Fourier expansion $y \sim \sum c_j x_j$ is Abel summable of order $\alpha > 1/\rho$ to $y$:
    \begin{equation}
    y = \lim_{t \to 0^+} \sum_{k=1}^\infty \left( \sum_{j=n_{k-1}}^{n_k-1} c_j(t) x_j \right),
    \end{equation}
    where the coefficients $c_j = \langle y, x_j^* \rangle$ are defined via the biorthogonal system $\{x_j^*\}$ of root vectors corresponding to the adjoint operator $T^*$. The Abel multipliers of order $\alpha$ are given by $c_j(t) = e^{-|\lambda_j|^\alpha t}$, and $(n_k)$ is a strictly increasing sequence of integers that partitions the spectrum into blocks of identical modulus to guarantee convergence.
\end{enumerate}
\end{theorem}

\begin{proof}
By Lemma 6.1, the condition $T \in A_\mu^\rho$ alongside the parameter restriction $0 < \rho \le 1/\mu$ guarantees that the sequence of approximation numbers $(a_n(T))$ resides within $l_p$ where $1/p = \rho$.

From Markus's fundamental bounds \eqref{eq:markus_ineq}, the eigenvalues are controlled by the approximation numbers:
\begin{equation}
|\lambda_n(T)| \le \frac{8C(C+1)}{2\sqrt{2}C} \delta_{n-1}(T) \le C' a_n(T).
\end{equation}
Since $(a_n(T)) \in l_p$, it is immediately guaranteed that the eigenvalue sequence satisfies $(|\lambda_n(T)|) \in l_p$, which implies $T \in \mathfrak{S}_p$.

With $T \in \mathfrak{S}_p$ and dense range $\overline{\mathfrak{R}(T)} = X$, Markus's Theorem $4.1^\circ$ directly guarantees that the system of root vectors is complete in $X$ (applying the theorem with $A = T$ and setting the perturbation operator to $0 \in \mathfrak{S}_\infty$ such that $B = T$).

To apply Markus's Theorem $3.3^\circ$ for Abel summability, we must explicitly verify its two hypotheses for $T$. First, the required growth limit (a) as $|\lambda| \to \infty$ is an entire-function order condition on the resolvent; because we have established $T \in \mathfrak{S}_p$, Markus's Theorem 2.3 directly guarantees that $(I-\lambda T)^{-1}$ satisfies this precise exponential type and order condition. Second, because $T$ is an $H$-operator, its spectrum $\sigma(T)$ is strictly real, meaning all non-zero eigenvalues lie on the real axis, trivially satisfying the angular confinement condition (b) for any angle $\Lambda$ containing the real line, with the $H$-operator resolvent bound $\|(T - \lambda I)^{-1}\| \le C|\mathrm{Im}\,\lambda|^{-1}$ securing the required boundedness on the angular boundaries. With both hypotheses verified and $T \in \mathfrak{S}_p$, the expansion is Abel summable for any order $\alpha > p$. Substituting $p = 1/\rho$, we conclude that the expansion is Abel summable of order $\alpha > 1/\rho$.
\end{proof}

\end{document}